\documentclass[11pt,twoside]{article}

\usepackage{mathrsfs,amsfonts,amsmath,amssymb}
\usepackage{latexsym}
\newtheorem{satz}{Theorem}

\newtheorem{theorem}[satz]{Theorem}
\newtheorem{lemma}[satz]{Lemma}

\newtheorem{corollary}[satz]{Corollary}
\newtheorem{remark}[satz]{Remark}

\def\Z{\mathbb {Z}}
\def\F{\mathbb {F}}
\def\E{\mathsf{E}}

\def\I{{\cal I}}
\def\a{\alpha}

\def\C{\mathbb{C}}
\def\P{{\cal P}}

\def\d{\delta}

\def\le{\leqslant}
\def\ge{\geqslant}
\def\_phi{\varphi}
\def\eps{\varepsilon}

\def\FF{\widehat}

\def\ov{\overline}
\def\Spec{{\rm Spec\,}}

\newtheorem{example}[satz]{Example}

\def\f{{\mathbb F}}

\def\la{\lambda}
\def\D{\Delta}

\def\C{\mathsf{C}}

\author{Shkredov I.D.}
\title{ A short note on the multiplicative energy of the spectrum of a set
}
\date{}
\begin{document}
	\maketitle

\begin{abstract}
	We obtain an upper bound for the multiplicative energy of the spectrum of an arbitrary set from  $\F_p$, 
	which is the best 
	possible up to the results on exponential sums over subgroups. 
\end{abstract}

\section{Introduction}
\label{sec:intr}

Let $p$ be a prime number and let $A$ be a subset of the prime field $\F_p$.
Denote by $\FF{A} (r)$, $r\in \F_p$ the Fourier transform of the characteristic function of the set $A$, namely,  
$$
	\FF{A} (r) = \sum_{a\in A} e^{-\frac{2\pi i ar}{p}} \,. 
$$
Given a real number  $\eps \in (0,1]$, define
\begin{equation}\label{def:spectrum}
\Spec_\eps (A) = \{ r \in \F_p ~:~ |\FF{A}(r)| \ge \eps |A| \} \,.
\end{equation}
The set $\Spec_\eps (A)$ is called the {\it spectrum} or the {\it set of large exponential sums} of our set $A$. 
Such sets are studied in \cite[Section 4.6]{TV}, further, in 
\cite{Bloom}---\cite{Green_Chang1},  \cite{Sanders_Roth_0.75}---\cite{S_diss} 
and in many other papers.
The spectrum appears naturally in any additive
problem  and, hence, 
it is important to know the structure of these sets. 
It is well--known that $\Spec_\eps (A)$ has strong {\it additive} properties, see, e.g., \cite{Bloom}, \cite{Ch_Fr}, \cite{S_les}.
This fact was used in \cite{S_sum-prod_les} to obtain a new property of the spectrum, namely, that  $\Spec_\eps (A)$ has poor {\it multiplicative} structure. 
It coincides with 
the philosophy 
 of the {\it sum--product} phenomenon, see, e.g., \cite{TV} that says that both additive and multiplicative structures  do not exist simultaneously.
Previously,  we used   the modern sum--product tools, see  \cite{Rudnev_pp}, \cite{RSS}
to demonstrate this poor multiplicative structure.
Here 
we apply the main sum--product result of \cite{Rudnev_pp} directly and obtain

\begin{theorem}
	Let $A\subseteq \F_p$ be a set, $|A| = \d p$ and $R \subseteq \Spec_\eps (A) \setminus \{0\}$ be any set. 
	Suppose that $p\le \eps^2 |A|^3$. 
	Then 
\begin{equation}\label{f:main_intr}
	|\{ (x,y,z,w) \in R^4 ~:~ xy=zw \}|
	\ll \eps^{-4} \d^{-1}  |R|^{3/2} \,.
\end{equation}
\label{t:main_intr}
\end{theorem}

Estimate \eqref{f:main_intr} is stronger than the results of \cite[Section 4]{S_sum-prod_les}  
and moreover one can show (see Remarks \ref{rem:E_4}, \ref{rem:E_2}) that the bound in Theorem \ref{t:main_intr} is sharp 
up to our current knowledge of some number--theoretical questions.
Also, in this paper we study other multiplicative characteristics of the spectrum, see Theorem \ref{t:no_R} and Theorem \ref{t:yes_R}, formula (\ref{f:yes_R_sigma}). 
As a byproduct we obtain by the same method a purely sum--product result, namely, a new lower bound on $AA+AA$ for sets with small sumset.

All logarithms are to base $2.$ The signs $\ll$ and $\gg$ are the usual Vinogradov symbols.
If we have a set $A$, then we will write $a \lesssim b$ or $b \gtrsim a$ if $a = O(b \cdot \log^c |A|)$, $c>0$.

\section{Notation and preliminary results}
\label{sec:notation}

	In this paper $p$ is an odd prime number, 
	$\F_p = \Z/p\Z$ and $\F_p^* = \F_p \setminus \{0\}$. 
	%
	%
	We denote the Fourier transform of a function  $f : \F_p \to \mathbb{C}$ by~$\FF{f},$
	\begin{equation}\label{F:Fourier}
	\FF{f}(\xi) =  \sum_{x \in \F_p} f(x) e( -\xi \cdot x) \,,
	\end{equation}
	where $e(x) = e^{2\pi i x/p}$. 
	We rely on the following basic identities. 
	The first one is called the Plancherel formula and its particular case $f=g$ is called the Parseval identity 
	\begin{equation}\label{F_Par}
	\sum_{x\in \F_p} f(x) \ov{g (x)}
	=
	\frac{1}{p} \sum_{\xi \in \F_p} \widehat{f} (\xi) \ov{\widehat{g} (\xi)} \,.
	\end{equation}
	Another  particular case of (\ref{F_Par}) is 
	\begin{equation}\label{svertka}
	\sum_{y\in \F_p} \Big|\sum_{x\in \F_p} f(x) g(y-x) \Big|^2
	= \frac{1}{p} \sum_{\xi \in \F_p} \big|\widehat{f} (\xi)\big|^2 \big|\widehat{g} (\xi)\big|^2 \,.
	\end{equation}
	In this paper we use the same letter to denote a set $A\subseteq \F_p$ and  its characteristic function $A: \F_p \to \{0,1 \}$.  
	Also, we write $f_A (x)$ for the {\it balanced function} of a set $A\subseteq \F_p$, namely, $f_A (x) = A(x) - |A|/p$.

	Let $A\subseteq \F_p$ be a set, and $\eps \in (0,1]$ be a real number.
	We have defined the set $\Spec_\eps (A)$ in \eqref{def:spectrum} already. 
	Clearly, $0\in \Spec_\eps (A)$, and $\Spec_\eps (A) = - \Spec_\eps (A)$.
	For further properties of $\Spec_\eps (A)$ see, e.g., \cite{Bloom}, \cite{Ch_Fr}, \cite{S_les}, \cite{S_sum-prod_les}. 
	Usually, 
	we denote by $\d$ the density of our set $A$, that is, $\d = |A|/p$.
	From Parseval identity (\ref{F_Par}), we have a simple upper bound for  size of the spectrum, namely,
	\begin{equation}\label{f:spec_Par}
	|\Spec_\eps (A)| \le \frac{p}{|A| \eps^2}  = \frac{1}{\d \eps^2} \,.
	\end{equation}

	Put
	$\E^{+}(A,B)$ for the {\it common additive energy} of two sets $A,B \subseteq \f_p$
	(see, e.g., \cite{TV}), that is, 
	$$
	\E^{+} (A,B) = |\{ (a_1,a_2,b_1,b_2) \in A\times A \times B \times B ~:~ a_1+b_1 = a_2+b_2 \}| \,.
	$$
	If $A=B$, then  we simply write $\E^{+} (A)$ instead of $\E^{+} (A,A)$
	and the quantity $\E^{+} (A)$ is called the {\it additive energy} in this case. 
	One can consider $\E^{+}(f)$ for any complex function $f$ as well.  
	Sometimes we  use representation function notations like $r_{AB} (x)$ or $r_{A+B} (x)$, which counts the number of ways $x \in \F_p$ can be expressed as a product $ab$ or a sum $a+b$ with $a\in A$, $b\in B$, respectively. 
	Put $\sigma^{+} (A) = \sum_{x \in A} r_{A-A} (x)$.  
	Further clearly
	\begin{equation*}\label{f:energy_convolution}
	\E^{+} (A,B) = \sum_x r_{A+B}^2 (x) = \sum_x r^2_{A-B} (x) = \sum_x r_{A-A} (x) r_{B-B} (x)
	\end{equation*}
	and by (\ref{svertka}),
	\begin{equation}\label{f:energy_Fourier}
	\E^{+}(A,B) = \frac{1}{p} \sum_{\xi} |\FF{A} (\xi)|^2 |\FF{B} (\xi)|^2 \,.
	\end{equation}
	Similarly, one can define $\E^\times (A,B)$, $\E^{\times} (A)$, $\E^{\times} (f)$  and so on.

	\bigskip

	Now we recall some results from 
	the incidence geometry, see, e.g., \cite[Section 8]{TV}.    
	First of all, we need a general design bound for the number of incidences, see \cite{s_D(A),TTT,Vinh}.
	Let $\mathcal{P} \subseteq \F_q^3$ be a set of points  and $\Pi$ be a collection of planes in $\F_q^3$. 
	Having $p\in \mathcal{P}$ and $\pi \in \Pi$, we write 
	\begin{displaymath}
	\I (p,\pi) = \left\{ \begin{array}{ll}
	1 & \textrm{if } q\in \pi\\
	0 & \textrm{otherwise.}
	\end{array} \right.
	\end{displaymath}
	Put $\mathcal{I} (\mathcal{P}, \Pi) = \sum_{p\in \P, \pi \in \Pi} \I (p,\pi)$.
	We have (see \cite{s_D(A)})

	\begin{lemma}
		For any functions $\a : \P \to \mathbb{C}$, $\beta : \Pi \to \mathbb{C}$ one has 
		\begin{equation}\label{f:Vinh}
		|\sum_{p,\pi} \I (p,\pi) \a(p) \beta(\pi) | \le q \| \a \|_2 \| \beta \|_2 \,,
		\end{equation}
		provided either $\sum_{p\in \P} \a(p) = 0$ or $\sum_{\pi\in \Pi} \beta (\pi) = 0$.
		\label{l:Vinh}
	\end{lemma}
	
	Of course, similar arguments work not just for points/planes incidences but, e.g., points/lines incidences and so on. 
	A much more deep 
	result on incidences is contained in \cite{Rudnev_pp} (or see \cite[Theorem 8]{MPR-NRS}).
	We formulate a combination of these results  and  
	Lemma \ref{l:Vinh}, see \cite{s_D(A)}.

	\begin{theorem}
		Let $p$ be an odd prime, $\mathcal{P} \subseteq \F_p^3$ be a set of points and $\Pi$ be a collection of planes in $\F_p^3$. 
		Suppose that $|\mathcal{P}| \le |\Pi|$ and that $k$ is the maximum number of collinear points in $\mathcal{P}$. 
		Then the number of point--planes incidences satisfies 
		\begin{equation}\label{f:Misha+_a}
		\mathcal{I} (\mathcal{P}, \Pi)  - \frac{|\mathcal{P}| |\Pi|}{p} \ll |\mathcal{P}|^{1/2} |\Pi| + k |\mathcal{P}| \,.	
		\end{equation}
		\label{t:Misha+}	
	\end{theorem}


Finally, we  
recall  a simple  asymptotic formula for the number of points/lines incidences in the case when  the set of points forms a Cartesian product,
see \cite{SZ_inc} and also \cite{s_D(A)}.

\begin{theorem}
	Let $A,B \subseteq \F_p$ be sets, $|A| \le |B|$, $\mathcal{P} = A\times B$, and $ \mathcal{L}$ be a collection of lines in $\F^2_p$.
	Then 
	\begin{equation}\label{f:line/point_as}
	\mathcal{I}(\mathcal{P}, \mathcal{L}) - \frac{|A| |B| |\mathcal{L}|}{p} \ll |A|^{3/4} |B|^{1/2} |\mathcal{L}|^{3/4} + |\mathcal{L}| + |A| |B| \,.
	\end{equation}
	\label{t:line/point_as}	
\end{theorem}

\section{The proof of the main results}
\label{sec:proofs}

	Let $A\subseteq \F_p$ be a set. We write 
$$
\E^\times_k (A) = \sum_x r^k_{A/A} (x) 
$$
for any $k\ge 1$. 
Our aim is to obtain an upper bound for $\E^\times_2$--energy of the spectrum but before that we prove an optimal result for  $\E^\times_4$ 
which is interesting in its own right.
We use arguments similar to \cite{Petridis_old}.

\begin{theorem}
	Let $A\subseteq \F_p$ be a set, $|A| = \d p$ and $R=\Spec_\eps (A) \setminus \{0\}$. 
	Then
\begin{equation}\label{f:no_R_E4}
	\E^\times_4 (R) \lesssim   \eps^{-16} \d^{-4} \left( \frac{\E^{+} (f_A)}{|A|^3} \right)^2 \,.
\end{equation}
\label{t:no_R}
\end{theorem}	
\begin{proof}
Applying formula \eqref{f:energy_Fourier} and the definition of the spectrum \eqref{def:spectrum}, we notice that
$$
	\frac{(\eps |A|)^4}{p} \cdot  r_{R/R} (\la) \le p^{-1} \sum_{x\in R,\, \la x \in R} |\FF{A} (x)|^2 |\FF{A} (\la x)|^2 \le p^{-1} \sum_{x} |\FF{f_A} (x)|^2 |\FF{f_A}
	 (\la x)|^2
	 = \E^{+} (f_A,\la f_A) \,.
$$
	Hence
$$
	\E^\times_4 (R) \le (\eps |A|)^{-16} p^4 \sum_{\la} \E^{+} (f_A,\la f_A)^4 =  (\eps |A|)^{-16} p^4  \sum_\la r^4_{(f_A-f_A)/(f_A-f_A)} (\la) = (\eps |A|)^{-16} p^4  \cdot \sigma \,.
$$
	By the Dirichlet principle 
	there is $\Delta>0$ and a set $P$ such that $\D < |r_{f_A-f_A} (\la)| \le 2\D$ on $P$ and 
$$
	\sigma \lesssim \D^4 \sum_\la r^4_{(f_A-f_A)/P} (\la) = \D^4 \sum_{\la} \left|\left\{ \la p = a_1-a_2 ~:~  p\in P \right\}\right|^4 \,,
$$
	where $a_1,a_2$ have weights $f_A(a_1), f_A (a_2)$, correspondingly. 
	Let $\tau>0$ and $S_\tau$ be the set of all $\la$ such that 
	$|r_{(f_A-f_A)/P} (\la)| \ge \tau$.
	Since $r_{(f_A-f_A)/P} (\la) = r_{(A-A)/P} (\la) +\d^2 p |P|$, it follows that
$$
	\tau |S_\tau| \le \sum_{\la \in S_\tau} |r_{(f_A-f_A)/P} (\la)| \le |A|^2 |P| + \d^2 |P| p^2  = 2|A|^2 |P| \,.
$$
	In particular, $|S_\tau| \le 2|A|^2 |P|/\tau$. 
	The number of the solutions to the equation $sp = a_1-a_2$ can be interpreted as the number of incidences between the set of  lines 
	$\mathcal{L} = S_\tau \times A$, counting with the weight $f_A(a_1)$ and the sets of points $\mathcal{P} = A\times P$, again counting with the weight $f_A(a_2)$. 
	Applying Theorem \ref{t:line/point_as}, we obtain
\begin{equation}\label{f:S_tau-}
	\tau |S_\tau| = 	\mathcal{I}(\mathcal{P}, \mathcal{L}) \ll 
	|A|^{3/2} |P|^{1/2} |S_\tau|^{3/4} + |S_\tau| |A| + |A| |P| \,.
\end{equation}
	If the first term dominates, then we have 
\begin{equation}\label{f:S_tau}
	|S_\tau| \ll |A|^6 |P|^2 /\tau^{4} \,.
\end{equation} 
	In view of the inequality $|S_\tau| \le 2|A|^2 |P|/\tau$ one can suppose that $\tau^3 \gg |A|^4 |P| \ge |A|^3$ because otherwise it is nothing to prove.
	It gives us that $\tau \gg |A|$ and hence the second term in \eqref{f:S_tau-} is negligible. 
	We will consider the case when the third term in \eqref{f:S_tau-} dominates later but now let us remark that in this case  $\tau^3 \gg |A|^5 |P|$ because otherwise it is nothing to prove.
	Thus, by summation of formula (\ref{f:S_tau}), we obtain 
$$
	\sigma \lesssim |A|^6 |P|^2 
$$
	and hence 
$$
	\E^\times_4 (R) \lesssim (\eps |A|)^{-16} p^4 |A|^6 \D^4 |P|^2 \le  \eps^{-16} \d^{-4} \E^{+} (f_A)^2 / |A|^6
$$
	as required.	
	It remains to consider the case when the third term in \eqref{f:S_tau-} dominates and we know that $\tau^3 \gg |A|^5 |P|$.
	In other words, if we consider the ordering 
	$$|r_{(f_A-f_A)/P} (s_1)| \ge |r_{(f_A-f_A)/P} (s_2)| \ge \dots \ge |r_{(f_A-f_A)/P} (s_j)| \ge \dots \,,$$
	then 
	there is an effective bound $|r_{(f_A-f_A)/P} (s_j)| \le |A| |P| j^{-1}$ for $j\ge J := (|P|/|A|)^{2/3}$. 
	Again, by summation we obtain
$$
	\sigma \ll \sum_{j\ge J} (|A| |P| /j)^4 \ll |A|^4 |P|^4 J^{-3} \ll |P|^2 |A|^6
$$
	and it gives the same bound for $\E^\times_4 (R)$. 
	This completes the proof. 
$\hfill\Box$
\end{proof}


\begin{remark}
	Let $A$ be a multiplicative subgroup of order $p^{2/3}$. 
	Then the best known bound for the Fourier coefficients of $A$ is $|\FF{A} (r)| <\sqrt{p}$, $\forall r\neq 0$, see,  e.g., \cite{KShp}. 
	On the other hand,  taking  $R$ equals a coset of $A$ belonging to $\Spec_\eps (A) \setminus \{0\}$,  we see that $\E^\times_4 (R) \ge |R|^5 = |A|^5$ . 
	Applying formulae \eqref{F_Par}, \eqref{f:energy_Fourier}, we get
	$$
	\E^{+} (f_A) < \left( \max_{r\neq 0} |\FF{A} (r)| \right)^2 |A| 
	$$
	and hence estimate \eqref{f:no_R_E4} of Theorem \ref{t:no_R} is tight (up to our current knowledge of the Fourier coefficients of multiplicative subgroups).
\label{rem:E_4}
\end{remark}

Unfortunately, the method of the proof of Theorem \ref{t:no_R} works for $\E_4^\times (R)$ but not for $\E_k^\times (R)$ with $k<4$. 
In this case we obtain

\begin{theorem}
	Let $A\subseteq \F_p$ be a set, $|A| = \d p$ and $R \subseteq \Spec_\eps (A) \setminus \{0\}$ be any set. 
	Suppose that $p\le \eps^2 |A|^3$. 
	Then
	\begin{equation}\label{f:yes_R_E2}
		\E^\times (R) \ll \eps^{-4} \d^{-1}  |R|^{3/2} \,.
	\end{equation}
	Similarly, 
	\begin{equation}\label{f:yes_R_sigma}
		\sigma^\times (R) \lesssim \eps^{-4} \d^{-1} |R|^{3/4} \left( \frac{\E^{+} (f_A)}{|A|^3} \right)^{1/2} 
		+ \eps^{-4} \d^{-1} \left( 1+ \frac{|R|}{|A|} \right)\,.
	\end{equation}
	\label{t:yes_R}
\end{theorem}	
\begin{proof}
	Using the Fourier transform similar to the proof of Theorem \ref{t:no_R}, we have 
$$
	\frac{(\eps |A|)^4}{p} \cdot  \E^{\times} (R) \le p^{-1} \sum_{\la, \mu \in R} \sum_x |\FF{f_A} (\la x)|^2 |\FF{f_A} (\mu x)|^2 
	=
	\sum_x r^2_{(f_A-f_A)R} (x) \,.
$$
	Clearly, the last quantity can be interpreted as points/planes incidences (with weights), see \cite{AMRS}. 
	Here the number  of the points and planes is at most $O(|A|^2 |R|)$.  
	Finally, using our assumption, we get from \eqref{f:spec_Par}
$$
	|R| \le \frac{p}{\eps^2 |A|} \le |A|^2  \,.
$$
	Applying Theorem \ref{t:Misha+}, we obtain
$$
	\sum_x r^2_{(f_A-f_A)R} (x) \ll |A|^3 |R|^{3/2} \,.
$$
	It follows that
$$
	\E^{\times} (R) \ll \eps^{-4} \d^{-1}  |R|^{3/2} \,.
$$
	as required.

	Similarly,
$$
	\sigma^\times (R) \le (\eps |A|)^{-4} \sum_{\la \in R} \sum_x |\FF{f_A} (x)|^2 |\FF{f_A} (\la x)|^2 
	= \eps^{-4} \d^{-1} |A|^{-3} \sum_{\la \in R} r_{(f_A-f_A)/(f_A-f_A)} (\la) \,.
$$
	After that we can use  the arguments and the notation  from the proof of Theorem \ref{t:no_R} (with $S_\tau = R$) and derive that 
$$
	\sum_{\la \in R} r_{(f_A-f_A)/(f_A-f_A)} (\la) \lesssim \D |P|^{1/2} |R|^{3/4} |A|^{3/2} + \D |A| (|P| + |R|) 
	\ll 
$$
$$
	\ll 
	(\E^{+} (f_A))^{1/2} |R|^{3/4} |A|^{3/2} 
	+ |A|^3 + |A|^2 |R| 
	\,.
$$
	Here we have used a trivial bound $\D \le 2|A|$. 
	It gives us 
$$
	\sigma^\times (R) \lesssim \eps^{-4} \d^{-1} |R|^{3/4} (\E^{+} (f_A) / |A|^3)^{1/2}  + \eps^{-4} \d^{-1} + \eps^{-4} \d^{-1} |R|/|A|
$$
	and this coincides with \eqref{f:yes_R_sigma}. 
$\hfill\Box$
\end{proof}

\begin{example}
	Let $\eps \gg 1$, $R = \Spec_\eps (A) \setminus \{0\}$, and let size of $R$ is comparable with the upper bound which is given by (\ref{f:spec_Par}),
	namely, $|R| \gg \d^{-1} \eps^{-2} \gg \d^{-1}$.
	Then $\E^{\times} (R) \lesssim |R|^{5/2}$.
	It means that we have a non--trivial estimate for the multiplicative energy of the spectrum in this case.
	Similarly, we always have $\E^{+} (f_A) < |A|^3$, so $\sigma^\times (R) \lesssim |R|^{7/4} + |R|^2/|A|$.
\end{example}

\begin{remark}
	The same construction as in Remark \ref{rem:E_4} shows the tightness of bounds \eqref{f:yes_R_E2}, \eqref{f:yes_R_sigma}, again up to our current knowledge of the Fourier coefficients of multiplicative subgroups. 
\label{rem:E_2}
\end{remark}

In the same vein we obtain a result on the growth of $AA+AA$, which improves \cite[Theorem 32]{s_D(A)} for small $\E_4^\times (A)$.

\begin{theorem}
	Let $A\subseteq \F_p$ be sets.
	Then
\begin{equation}\label{f:Misha_zero_sum}
	\sum_{x} r^2_{AA+AA} (x) - \frac{|A|^8}{p}  \lesssim  |A|^4 (\E^{\times}_4 (A) )^{1/2} + \E^{\times}_4 (A) |A|^2 \,.
\end{equation}
\label{t:Misha_zero_sum}
\end{theorem}
\begin{proof}
	Without loosing of the generality, one can assume that $0\notin A$. 
	We need to estimate the number of the solutions to the equation
$$
	a_1/a \cdot a'_1/a' + a_2/a \cdot a'_2/a' - a_3 / a \cdot a'_3/a' = 1 \,, 
$$
	where $a,a', a_j, a'_j \in A$, $j=1,2,3$. 
	Put
$$
	\C^\times_4 (A) (\a,\beta,\gamma) := |A\cap \a A \cap \beta A \cap \gamma A| \,.
$$
	One can check that
$$
	\sum_{\a,\beta,\gamma} 	\C^\times_4 (A) (\a,\beta,\gamma) = |A|^4 \,,
$$
	and 
\begin{equation}\label{f:C_E}
	\sum_{\a,\beta,\gamma} 	\C^\times_4 (A) (\a,\beta,\gamma)^2 = \E^\times_4 (A) \,.
\end{equation}
	In these terms, 
	we want to bound the sum 
$$
	\sigma := 
	\sum_{\a,\beta,\gamma}\, \sum_{\a',\beta',\gamma'} \C^\times_4 (A) (\a,\beta,\gamma) \C^\times_4 (A) (\a',\beta',\gamma') 
		\d (\a \a'+\beta\beta'-\gamma\gamma' = 1) \,,
$$
	where $\d(x=1)$ equals one iff $x=1$. 
	Using the Dirichlet principle as in the proof of Theorems \ref{t:no_R}, \ref{t:yes_R}, 	we find two numbers $\D_1, \D_2 >0$ 
	and two corresponding  sets of points and planes
	$\mathcal{P}$, $\Pi$ 
	such that 
$$
	\sigma \lesssim \D_1 \D_2 \sum_{\a,\beta,\gamma}\, \sum_{\a',\beta',\gamma'} \mathcal{P} (\a,\beta,\gamma) \Pi (\a',\beta',\gamma')  \d (\a \a'+\beta\beta'-\gamma\gamma' = 1) \,.
$$
	Without loosing of the generality, suppose that $|\mathcal{P}| \le |\Pi|$. 
	Also, notice that $|\mathcal{P}|, |\Pi| \le |A|^4$. 
	Applying Theorem \ref{t:Misha+} 
	(previously inserting  the balanced function $f_A(x) = A(x) - |A|/p$ as in the proofs of Theorems \ref{t:no_R}, \ref{t:yes_R})
	with the maximal number of collinear points $k\le |A|^2$ and using formula \eqref{f:C_E}, combining with Lemma \ref{f:Vinh}, we
	get 
$$
	\sigma \lesssim \D_1 \D_2 |\mathcal{P}| |\Pi|^{1/2} + \D_1 \D_2 k |\mathcal{P}| + |\mathcal{P}|^{1/2} \E^\times_4 (f_A) 
	\le
$$
$$
	\le (\D_2^2 |\Pi|)^{1/2} \D_1 |\mathcal{P}| + k (\D^2_1 |\mathcal{P}|)^{1/2} (\D^2_2 |\Pi|)^{1/2} + |A|^2 \E^\times_4 (f_A)
	\le
$$
$$ 
	\le
	(\E^{\times}_4 (A))^{1/2} |A|^4 + \E^{\times}_4 (A) |A|^2 \,.
$$
	This completes the proof. 
$\hfill\Box$
\end{proof}

\begin{corollary}
	Let $A\subseteq \F_p$, $|A+A| = K|A|$ and $|A+A|^3 |A| \le p^3$.
	Then
$$
	|AA+AA| \gtrsim \min \{ p, \Omega_K (|A|^2) \} \,.
$$
\label{c:AA+AA}
\end{corollary}
\begin{proof}
	Using \cite[Lemma 18]{MPR-NRS}  (where, actually, a better dependence on $K$ is suggested)
	or just applying  the arguments of the proof of Theorem \ref{t:no_R}, we get
$$
	\E^\times_4 (r_{B+C}) - \frac{|B|^8 |C|^8}{p^3} \lesssim  \E^{+} (B,C)^2 |B|^3 |C|^3 \,.
$$
	Putting $B=A+A$, $C=-A$ and noting that $|A| A(x) \le r_{B+C} (x)$, we obtain
\begin{equation}\label{tmp:E_4_A+A_small}
	\E^\times_4 (A) - \frac{|A+A|^8}{p^3} \lesssim   |A+A|^5 |A|^{-1} \,.
\end{equation}
	Obviously, by the Cauchy--Schwartz inequality,  we have
\begin{equation}\label{tmp:CS} 
	|A|^8 \le |AA+AA| \cdot \sum_{x} r^2_{AA+AA} (x) \,.
\end{equation}
	By Theorem \ref{t:Misha_zero_sum}, we get 
$$
	\sum_{x} r^2_{AA+AA} (x)  - \frac{|A|^8}{p} \lesssim  |A|^4 (\E^{\times}_4 (A) )^{1/2} + \E^{\times}_4 (A) |A|^2 \,.
$$
	If the term $\frac{|A|^8}{p}$  dominates in the last formula, we have from \eqref{tmp:CS} that  $|AA+AA| \gg p$.
	Otherwise in view of \eqref{tmp:E_4_A+A_small} and our condition  $|A+A|^3 |A| \le p^3$, we see that 
$$ 
	|A|^8 \lesssim |AA+AA| \cdot \left( |A|^4 (\E^{\times}_4 (A) )^{1/2} + \E^{\times}_4 (A) |A|^2 \right) 
	\ll |AA+AA| \cdot  K^5 |A|^6  \,.
$$
	This completes the proof. 
$\hfill\Box$
\end{proof}

\bigskip 

Considering $A=\{1,2,\dots, n\}$, where $n$ is sufficiently small comparable to $p$, we see that Corollary \ref{c:AA+AA} is the best possible up to logarithms.


\bigskip

\noindent{I.D.~Shkredov\\
Steklov Mathematical Institute,\\
ul. Gubkina, 8, Moscow, Russia, 119991}
\\
and
\\
IITP RAS,  \\
Bolshoy Karetny per. 19, Moscow, Russia, 127994\\
and 
\\
MIPT, \\ 
Institutskii per. 9, Dolgoprudnii, Russia, 141701\\
{\tt ilya.shkredov@gmail.com}

\end{document}